\documentclass[10pt,notitlepage]{article}
\pdfoutput=1
\usepackage[latin1]{inputenc}
\usepackage[hmargin=1.25in,vmargin=1.25in]{geometry}
\usepackage[noend]{algorithm2e}
\usepackage{amsmath}
\usepackage{amsfonts}
\usepackage{amssymb}
\usepackage{amsthm}
\usepackage{caption}
\usepackage{float}
\usepackage{graphicx}
\usepackage{subcaption}
\newtheorem{thm}{Theorem}
\newtheorem{lem}[thm]{Lemma}
\begin{document}

\title{Exploration of Network Scaling: Variations on Optimal Channel Networks}
\author{Lily Briggs \and Mukkai Krishnamoorthy}
\maketitle

\begingroup
\let\clearpage\relax

\begin{abstract}
Metabolic allometry, a common pattern in nature, is a close-to-3/4-power scaling law between metabolic rate and body mass in organisms, across and within species.  An analogous relationship between metabolic rate and water volume in river networks has also been observed.
Optimal Channel Networks (OCNs), at local optima, accurately model many scaling properties of river systems, including metabolic allometry.
OCNs are embedded in two-dimensional space; this work extends the model to three dimensions.
In this paper we compare characteristics of 3d OCNs with 2d OCNs and with organic metabolic networks, studying the scaling behaviors of area, length, volume, and energy.
In addition, we take a preliminary look at comparing Steiner trees with OCNs.
We find that the three-dimensional OCN has predictable characteristics analogous to those of the two-dimensional version, as well as scaling properties similar to metabolic networks in biological organisms.
\end{abstract}

\section{Introduction}

Fractal Networks have been widely used to study well known network behaviors.  Some of the properties commonly used in the study of fractal networks are scaling factors (or laws); scaling laws can describe how lengths
are related to area and areas are related to volumes. In this paper, we further study these allometric scaling laws.  We use simulation and elementary mathematical bounding techniques to study such laws.

A common pattern in nature, known as metabolic allometry, is a 3/4-power scaling law between metabolic rate ($M$) and body mass ($B$) in organisms, across and within species: $B \propto M^{3/4}$.  An analogous relationship between water mass ($C$) and contributing area ($A$) in river networks has also been observed: $A \propto C^{2/3}$.

The relationship in biological organisms was first observed by Max Kleiber in 1932 and has garnered a lot of attention in the last decade and a half from biologists, mathematicians, and physicists alike \cite{Kleiber1932Body}.  The areas this relationship finds application in include models of climate change effects in rivers \cite{Renner2012Evaluation} and oceans \cite{ocean_balance}, analyses of river stability \cite{Molnar1998Analysis}, and in studies of information networks \cite{Moses2008Scaling}.  Many theories concerning the origins and nature of this phenomenon have been proposed, and there has been much debate about how ubiquitous the scaling law really is.

In 1997, West, Brown and Enquist (WBE hereafter) introduced a theory purporting to explain metabolic scaling \cite{wbe}. Their theory is based on the idea of a fractal space-filling hierarchical network of resource distribution, with certain assumptions based on simplified biology.  The geometry of the network (e.g. branching ratios, vessel radius relationships) are important to the development of their result. The authors claimed the model is applicable even to organisms without a physical blood vessel network since they can be treated as having a virtual distribution network \cite{West1999Fourth}.

Subsequent studies questioned the WBE results, by showing that data does not appear to strictly follow a 3/4-power scaling law after all; the theoretical approach taken by WBE has also been criticized \cite{Agutter2011Analytic, White2011Manipulative}.  Some authors claim that an exponent of 2/3 fits the data better \cite{Dodds2001Reexamination}, while others assert that no universal scaling law exists at all \cite{Kolokotrones2010Curvature}.
Savage et al. pointed out that the WBE model predicts a curvilinear relationship between $log B$ and $log M$, and only predicts a constant power-law relationship in the infinite limit of body size; however, even with finite size adjustments and variations in the structural properties of the model, they found the results were still inconsistent with data \cite{Savage2008Sizing}.
Some other modifications to the WBE model do yield results somewhat more consistent with data, such as taking into account body temperature \cite{Brown2004Toward} or fluid velocity \cite{Banavar2010General}.
Kolokotrones et al. use regression on large set of organism data to show a quadratic, rather than linear, relationship between $log B$ and $log M$ \cite{Kolokotrones2010Curvature}. Also, their results reported a lot of variation.  They build off the modifications to WBE in \cite{Savage2008Sizing} and find that a model network with a proportional transition between area-increasing and area-preserving branching yields a fit to the data almost as good as their empirically-derived quadratic model.

Other theories and models have been put forth, some based on network properties like the WBE model, and others on other biological properties.  Dodds \cite{Dodds2010Optimal} proposed a model based on virtual networks of metabolite transportation rather than physical networks.

Banavar et al. proposed a general model for any efficient transportation network, predicting $B \propto M^{D/(D+1)}$ (at most) where $D$ is the dimension of the network (i.e. $D = 3$ for organism metabolism, and $D = 2$ for river networks) \cite{size_and_form}.  Thus, they predict that if the network were as efficient as possible, organism metabolism should scale as $B \propto M^{3/4}$, and deviations of the exponent below $3/4$ might be explained by some inefficiency in the metabolic distribution network.  Later the authors added to this model a supply and demand mechanism that accounts for the deviations \cite{Banavar2002Supplydemand}.

Isaac et al. propose a less strict understanding of the ``3/4-power law'': though the 3/4 exponent might not be a consistent, universal constant, it is still an observable trend; if there isn't a single unifying principle explaining why it appears, there still may be multiple reasons for it worth studying \cite{Isaac2010Why}.

In the context of river networks, the analogous properties to metabolic rate and body mass are contributing area (A) and water mass (C).  It is well-observed that the properties obey a power-law relationship with an exponent of slightly more than 2/3: $A \propto C^{2/3}$.  This is consistent with the general resource transportation model of Banavar et al. \cite{size_and_form} for systems in two dimensions, and also with the virtual network model of Dodds \cite{Dodds2010Optimal}.

Rinaldo et al. devised a model of river networks called Optimal Channel Networks (OCNs) \cite{Rinaldo1992Minimum}.  This model uses a grid network to represent the area of a river basin, and constructs a spanning network on the grid that minimizes a functional representing the energy required to drain the area of land through river channels.  The energy functional is derived both from theoretical physical characteristics and observed characteristics of river basins \cite{RodriguezIturbe1997Optimal}.  Statistical properties of Optimal Channel Networks (including metabolic allometry) correspond excellently with observed properties of real river networks \cite{Maritan2002, RodriguezIturbe2011Metabolic}.

This paper explores a three-dimensional version of OCNs, using simulations to compare characteristics of 3d OCNs with 2d OCNs and with organic metabolic networks and study the scaling behavior of area, length, volume, and energy.  We do not attempt to explain 3/4-power scaling in nature, but rather explore some interesting and potentially useful connections, somewhat in the spirit of Isaac et al.'s proposed understanding of the scaling law \cite{Isaac2010Why}.

We further study an alternative model using Steiner trees for OCNs and we compare the scaling behaviors between these two models on small network sizes.

Section 2 describes definitions and methodologies used in our paper.
In Section 3, we derive lower bound results for various scaling laws.
In section 4, we describe our simulation results and show Energy, Length and Volume Scaling with respect to area
and compare it with the lower bound results described in Section 3. In Section 5, we describe an alternative model based on Steiner Trees and derive analytical results for small network sizes. Section 6 gives conclusions and suggestions for future work.

\section{Definitions and Methods}

Consider an area of land, and embed a grid network in it such that every node in the network is associated with a unit area of land.  Construct a spanning tree on the grid and direct it so that all nodes have a path in the tree to a single root (call it the outlet).  Define a partial order on the nodes such that $x \le y $ iff $x$ is on the path from $y$ to the root.  If $x < y$, we will say $y$ is upstream of $x$.  The terms area, volume, upstream length, link length, and energy are defined here in this context.

Note that each node in the spanning tree, except for the root, has exactly one link directed out of it.  Thus we can name each link according to its corresponding node; that is, when we speak of link $x$ we are referring to the link out of node $x$.  All the properties defined here for nodes will be extended to their corresponding links.  For example, the area of link $x$ is defined to be the area of node $x$.

The area of node $x$, $A_x$, is the total number of nodes whose paths to the outlet include $x$; that is, the number of nodes $y$ such that $x \le y$.  This is equivalent to the sum of the areas of each node directly linked to $x$ plus one for $x$ itself:

\begin{equation} \label{eq:area}
	A_x \:=\: \sum_{y} A_y \:+\: 1
\end{equation}
where y ranges over the nodes directly linked to x \cite{Rinaldo1992Minimum}.

The volume of node $x$, $C_x$, is the sum of all the areas of all nodes upstream of $x$ \cite{size_and_form}:
\begin{equation} \label{eq:volume}
	C_x \:=\: \sum_{y | x < y} A_y
\end{equation}

Equation \ref{eq:volume} can be transformed into a recursive definition for volume:

\begin{equation} \label{eq:recvol} 
	 C_x \:=\: \sum_{y} (C_y + A_y)
\end{equation}
where y ranges over the nodes directly linked to x. 

Upstream length is another quantity of interest; it is defined as the number of links in the path from a node $x$ to a source node, where at each step upstream the path taken goes to the node with the largest area \cite{Cieplak1998Models}.  Statistically, it is equivalent to define it as the distance (counted in number of links) to the farthest away upstream source \cite{Rinaldo2006Trees}.

Link length is a weight assigned to each link, generally representing the geometric length of the link in the embedding.  The length of link $i$ is denoted by $l_i$.

Energy is a function of the link lengths and areas of all the links.  The energy of a given tree configuration $s$ is given by:
\begin{equation} \label{eq:energy}
	H_{\gamma}(s) = \sum_{i \in links(s)} {A_i^{\gamma} l_i}
\end{equation}
with $0 \le \gamma \le 1$.  A tree network that achieves the minimum $H_{\gamma}(s)$ over all possible tree configurations on a given grid is an optimal channel network \cite{Rinaldo1992Minimum}.

Figure \ref{fig:ocnex} shows a grid and a possible tree configuration $s_1$.  $s_1$ minimizes $H_{1/2}(s)$ for all configurations on this grid.  Area, volume and upstream length values are also depicted for the nodes in $s_1$.

\begin{figure}[H]
	\centering
	\begin{subfigure}[c]{0.2\textwidth}
		\centering
		\includegraphics[scale = 0.5]{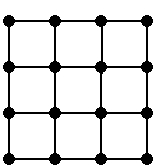}
		\caption{Underlying grid \\ { } }
		\label{fig:ocnexgrid}
	\end{subfigure} \hfill
	\begin{subfigure}[c]{0.2\textwidth}
		\centering
		\includegraphics[scale = 0.5]{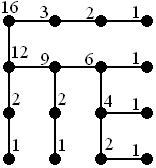}
		\caption{Areas of nodes \\ in $s_1$}
		\label{fig:ocnexarea}
	\end{subfigure} \hfill
	\begin{subfigure}[c]{0.2\textwidth}
		\centering
		\includegraphics[scale = 0.5]{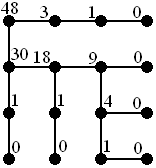}
		\caption{Volumes of nodes \\ in $s_1$}
		\label{fig:ocnexvolume}
	\end{subfigure} \hfill
	\begin{subfigure}[c]{0.2\textwidth}
		\centering
		\includegraphics[scale = 0.5]{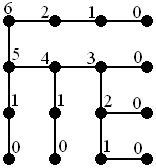}
		\caption{Upstream lengths \\ of links in $s_1$}
		\label{fig:ocnexupl}
	\end{subfigure} \hfill
	\caption{The outlets are in the top left corners.  Tree $s_1$ minimizes $H_{1/2}(s)$ for the grid in \ref{fig:ocnexgrid}. $H_{1/2}(s_1) = 24.3024964149$.}
	\label{fig:ocnex}
\end{figure}

Rinaldo et al. \cite{Rinaldo1992Minimum} derived a gamma value of 1/2 from estimations of physical properties of rivers, and hypothesized that natural rivers would minimize $H_{1/2}(s)$.  A remarkable result of \cite{Rinaldo1992Minimum} is that while the global optima of OCNs do not match well statistically with real river networks, the local optima are an excellent match.  The idea that nature tends towards a local optimum makes sense, so this suggests that OCNs are a good model for what actually happens in a river \cite{RodriguezIturbe1997Optimal}. 

When $\gamma = 1$ and all links are of unit length, the energy is equivalent to the volume.  When $\gamma = 0$, the energy is equivalent to the total length of all links, so minimizing the energy is equivalent to finding a minimum spanning tree (which, when all links are of unit length, would be \emph{any} spanning tree).


In most studies of OCNs, the network of potential links is a 2-dimensional lattice with additional links between diagonally-adjacent pairs of nodes (Fig. \ref{fig:2dgrid}).  For the sake of simplicity, all the links, diagonal and orthogonal, are here considered to have unit length, because it has been shown (\cite{RodriguezIturbe1997Optimal}) that the properties with which this paper is concerned are independent of whether the diagonal links are given realistic length or not.

\begin{figure}[H]
	\begin{subfigure}[b]{0.5\textwidth}
		\centering
		\includegraphics[scale = .5]{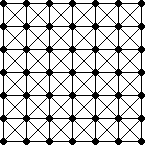}
		\caption{7x7 network of potential links}
		\label{fig:2dgrid}
	\end{subfigure}
	\begin{subfigure}[b]{0.5\textwidth}
		\centering
		\includegraphics[scale = .5]{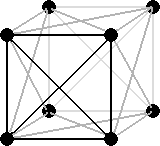}
		\caption{2x2x2 3d network of potential links}
		\label{fig:3dgrid}
	\end{subfigure}
	\caption{}
\end{figure}

In this study we also analyzed a three dimensional version of optimal channel networks, built on a 3-dimensional grid (Fig. \ref{fig:3dgrid}).  

To make an OCN, we first generated a random spanning tree on a given grid.  We used Prim's algorithm for minimum spanning trees for this.  Essentially, the algorithm builds the tree from a root node, maintaining a set of frontier edges from the grid that are incident on exactly one node in the tree.  Our algorithm randomly chooses one edge at a time from this set, adds it to the tree, and updates the set.  Rinaldo et al. also began with random spanning trees, though they did not specify their method for generating them \cite{Rinaldo1992Minimum}.

The random tree was then optimized using a version of Lin's algorithm for TSP, as in \cite{Rinaldo1992Minimum}.  That algorithm is as follows: \medskip

\begin{algorithm}[H]
\nonumber
\SetAlgoLined
\KwIn{Undirected graph $G$; random spanning tree $T$, directed so that there exists a path in $T$ from each node to a single root}
\KwOut{Locally optimal spanning tree on $G$}

\While{Convergence has not been reached}{
Step 1: Choose a random node $i$ in $T$

Step 2: Choose a random neighbor of $i$ in $G$ distinct from the node $i$

\Indp \Indp links to in $T$

\Indm \Indm Step 3: Construct $T^*$ from $T$ by redirecting $i$ to link to this new node

\lIf{$T^*$ contains a loop}{
	discard $T^*$
}

\lElseIf{$H_{\gamma}(T^*) < H_{\gamma}(T)$}{
	discard $T$ and let $T^*$ be the new $T$
}

\lElse{discard $T^*$}

}
\label{alg:mine}
\end{algorithm}

The condition for convergence was that the number of improvements made be only 1\% of the total number of iterations (or 2\% for larger networks).  See Figure \ref{fig:steps} for an illustration of the process.

\begin{figure}[H]
	\begin{subfigure}[b]{0.5\textwidth}
		\centering
		\includegraphics[scale = .6]{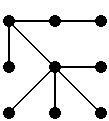}
		\caption{Initial tree configuration $T$.  The outlet is in the top left corner. $H_{\gamma}(T) = 9.65028$.}
	\end{subfigure} \quad
	\begin{subfigure}[b]{0.5\textwidth}
		\centering
		\includegraphics[scale = .6]{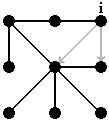}
		\caption{Step 1: Node $i$ has been chosen; the unused links to its neighbours in $G$ are highlighted.}
	\end{subfigure} \\
	\begin{subfigure}[b]{0.5\textwidth}
		\centering
		\includegraphics[scale = .6]{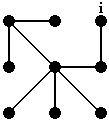}
		\caption{Step 3: A new link for $i$ in the tree has been chosen, creating $T^*$.  $H_{\gamma}(T^*) = 9.8637 > H_{\gamma}(T)$, so this tree will be discarded.}
	\end{subfigure} \quad
	\begin{subfigure}[b]{0.5\textwidth}
		\centering
		\includegraphics[scale = .6]{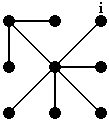}
		\caption{the next iteration, Step 3:  A different new link for $i$ in the tree has been chosen, creating $T^*$.  $H_{\gamma}(T^*) = 9.44949 < H_{\gamma}(T)$, so this tree will become the new $T$.}
	\end{subfigure}
	\caption{Optimization algorithm}
	\label{fig:steps}
\end{figure}

\begin{figure}[H]
	\begin{subfigure}[b]{0.5\textwidth}
		\centering
		\includegraphics[scale = .4]{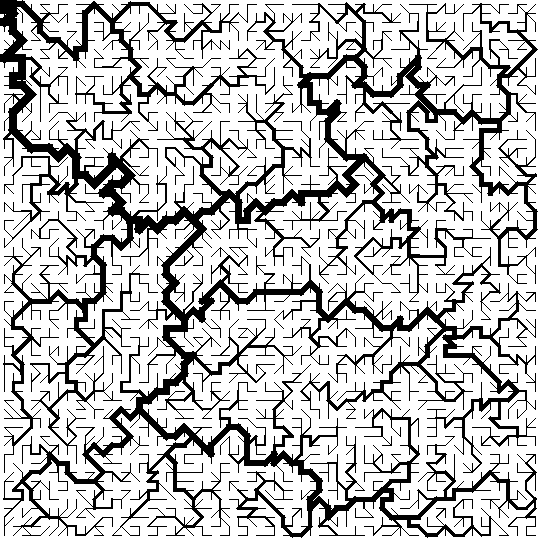}
		\caption{}
		\label{fig:ocnbefore}
	\end{subfigure}
	\begin{subfigure}[b]{0.5\textwidth}
		\centering
		\includegraphics[scale = .4]{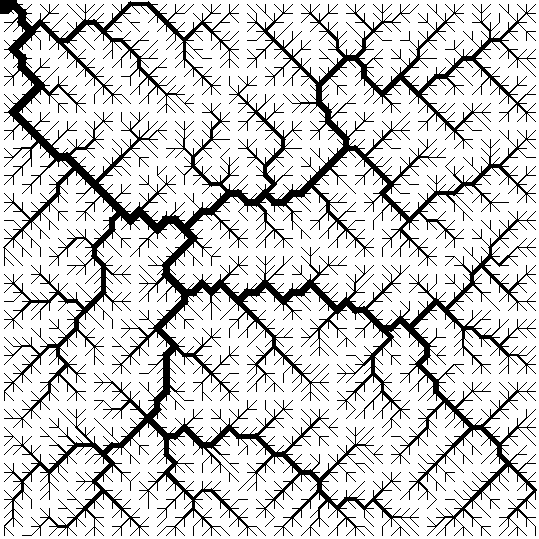}
		\caption{}
		\label{fig:ocnafter}
	\end{subfigure}
	\caption{A random tree spanning a 60x60 grid (a) and the resulting locally optimal OCN (b).  The outlets are in the top left corners.}
\end{figure}

\begin{figure}[H]
	\centering
	\includegraphics[scale = .5]{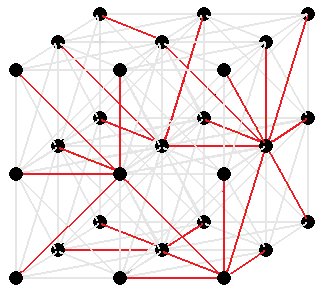}
	\caption{An OCN on a 3x3x3 grid; the red links are in the OCN.}
	\label{fig:3docn}
\end{figure}

Figure \ref{fig:ocnbefore} shows a random tree on a 60x60 grid and Figure \ref{fig:ocnafter} shows the result of using the above algorithm to optimize the tree.  In order to make the structure of the network clearer, links are drawn with a thickness proportional to the log of their area.  Figure \ref{fig:3docn} depicts a three-dimensional OCN on a 3x3x3 grid.

A range of grid sizes were analyzed.  Two-dimensional $n$x$n$ grids with $n =$ $10$, $20$, $30$, $40$, $50$, $60$, $70$, and $80$, and three-dimensional $n$x$n$x$n$ grids with $n = 8$, $10$, $12$, $14$, $16$, $18$, and $20$ were used.  For the smaller grids, data was averaged from at least $4$ realizations of each size, while for the larger grids at least $3$ realizations were used for each.  The quantities measured were distributions of area and length in whole basins, and average length, volume, and energy per area of subbasins and whole basins.

\section{Analytical Proofs} \label{sec:proofs}

\subsection{Energy bounds}

In \cite{Colaiori1997Analytical}, Colaiori et al. proved lower bounds for the energy of OCNs on an orthogonal grid (one with no diagonal links).  Here, we look at lower bounds for the energy of OCNs on an eight-neighbor grid.  The main result for this section (Theorem 3) is that for an $n \times n$ 8-way grid, $\frac{3}{2}n^2 - \frac{7}{2}n + 1$ is a lower bound on $H_{0.5}(s)$.

\begin{lem} \label{lem3} Let (P) be the optimization problem
	
	\[ \begin{array}{rl}
		\min        & \displaystyle\sum_{i = 1}^n A_i^{\gamma} \\
		\mbox{s.t.} & \displaystyle\sum_{i = 1}^n A_i = n + m \\
					& A_i \ge 1 \:\: \forall i \\
					& A_i \in \mathbb{Z} \:\: \forall i
		\end{array}
	\]
	
	where $n$ is the number of entries in the vector $A$, $m$ is some natural number, and $0 \le \gamma \le 1$. \\
	
	$A_1 = A_2 = ... = A_{n-1} = 1, A_n = m+1$ is an optimal solution to $(P)$.
\end{lem}

\begin{proof} 
Suppose A is an optimal solution for $(P)$ and $\exists j$ s. t. $1 < A_j < m + 1$.

$\exists k \:\: s. t. \:\: 1 < A_k < m + 1, j \ne k$.

We can construct $A^*$ such that $A^*_i = A_i \:\: \forall i \ne j, k$; $A^*_j = 1$; $A^*_k = A_j + A_k - 1$.

Then $\sum_{i = 1}^n {A^*_i}^{\gamma} = (\sum_{i = 1}^n A_i^{\gamma}) - A_j^{\gamma} - A_k^{\gamma} + 1 + (A_j + A_k - 1)^{\gamma}$.

Since $f(x)= x^\gamma$ when $0 \le \gamma \le 1$ is concave, and $A_j, A_k > 1$, $(A_j + A_k - 1)^{\gamma} \le A_j^{\gamma} + A_k^{\gamma} - 1$.  So $\sum_{i = 1}^n {A^*_i}^{\gamma} \le \sum_{i = 1}^n A_i^{\gamma}$.  Thus $A^*$ must be an optimal solution for $(P)$.  By an iteration of this argument, there must be an optimal solution $A^+$ with all entries except for one equal to $1$.  Without loss of generality, the entries can be rearranged so that $A^+_n$ is the largest (only non-unit) entry.
\end{proof} \medskip

\begin{thm} For an $n \times n$ eight-neighbor grid $G$,  the optimal $H_1(s)$ is $\frac{4n^3 - 3n^2 - n}{6}$. \end{thm}
\begin{proof}

\begin{figure}[h] 
	\centering
	\includegraphics[scale = .5]{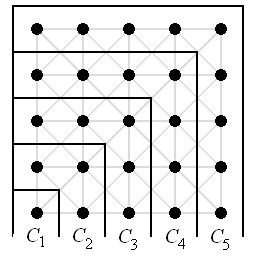}
	\caption{The five ``stripes'' of a $5 \times 5$ 8-way grid with the outlet in the bottom left corner.}
	\label{fig:stripes}
\end{figure}

Consider the subsets of vertices $C_k$ where vertex $i \in C_k$ iff the distance from $i$ to the outlet equals $k-1$ (Figure \ref{fig:stripes}).  In an $n \times n$ 8-way grid there are n such stripes.  Assuming the outlet is located in a corner, $|C_k| = 2k - 1$.
Clearly,

\begin{equation} \label{eq:doublesum}
H_\gamma(s) = \sum_{i \in links(s)} A_i^\gamma = \sum_{k=2}^n \sum_{i \in C_k} A_i^\gamma
\end{equation}

Recall that the area of node $i$ is the number of nodes upstream of $i$, plus $1$ for $i$ itself.  So for a given $k$, $\sum_{i \in C_k} A_i^\gamma$ will have to include the number of nodes in $C_k$ and the total number of nodes upstream of all $i \in C_k$, which must include all the nodes $j \in \bigcup_{l > k} C_l$, since their paths to the outlet must pass through stripe $C_k$ at some point.  $\bigcup_{l > k} C_l$ is simply the set of all the nodes in the grid minus those within the $k$x$k$ section enclosed by $C_k$, which is $n^2 - k^2$ nodes.  Thus the total amount of area being passed through $C_k$ is at least
\begin{equation} \label{qu:stripearea}
|C_k| + n^2 - k^2
\end{equation}

Then we have

\[H_1(s) \: =\: \sum_{k=2}^n \sum_{i \in C_k} A_i \:\ge\: \sum_{k=2}^n ( 2k - 1 +  n^2 - k^2 ) \]

The rightmost sum simplifies to $\frac{1}{6}(4n^3 - 3n^2 - n)$; therefore this is a lower bound on $H_1(s)$.

Consider the network $t$ where every link from node $i$ to node $j$ is such that if $i \in C_k$ then $j \in C_{k-1}$.  Here, quantity \ref{qu:stripearea} is exactly the area passing through stripe $C_k$, so $H_1(t) = \frac{1}{6}(4n^3 - 3n^2 - n)$.  Hence the lower bound is achievable.

\end{proof} \medskip

\begin{thm} \label{thm:main} For an $n \times n$ eight-neighbor grid, $\frac{3}{2}n^2 - \frac{7}{2}n + 1$ is a lower bound on $H_{0.5}(s)$. \end{thm}
\begin{proof}

Following the same reasoning as above, we have that

\begin{equation}
H_\gamma(s) = \sum_{k=2}^n \sum_{i \in C_k} A_i^\gamma
\end{equation}

and the total amount of area being passed through $C_k$ is at least $|C_k| + n^2 - k^2$.

By Lemma \ref{lem3}, the optimal way to distribute this area over the vertices in $C_k$ is to send all the upstream area, $n^2-k^2$, through one vertex in $C_k$, leaving each of the other vertices in $C_k$ with area 1.  (Note that this is not usually feasible, but it is a lower bound.) \\*
Thus for a given $k$,
\[ \sum_{i \in C_k} A_i^\gamma \ge |C_k| - 1 + (n^2 - k^2 + 1)^\gamma \]
Since $f(x) = x^\gamma$ is concave,
\[\sum_{k=2}^n (n^2 - k^2 + 1)^\gamma \:\ge\: \sum_{k=2}^n [n^{2\gamma} - (k^2 - 1)^\gamma] \:\ge\: n^{2\gamma + 1} - n^{2\gamma} - \sum_{k=2}^n (k^2)^\gamma  \]
So
\[ \sum_{k=2}^n \sum_{i \in C_k} A_i^\gamma
 \:\ge\: \sum_{k=2}^n [2k-2 + (n^2 - k^2 + 1)^\gamma]
  \:\ge\: n^2 - 3n + n^{2\gamma +1} - n^{2\gamma} - \sum_{k=2}^n k^{2\gamma}  \]
For $\gamma = 0.5$, this evaluates to $2n^2 - 4n - (\frac{n^2-n}{2} - 1) \:=\: \frac{3}{2}n^2 - \frac{7}{2}n + 1$.
\end{proof}


\medskip

\subsection{Volume Scaling}

In \cite{Maritan2002}, Maritan et al. make an analytical prediction for the scaling of volume with area in rivers.  Their analysis depends on the relationship $ \langle L_x \rangle \propto A_x^h$, where $ \langle L_x \rangle $ is the mean distance from vertices upstream of $x$ to $x$, and $h$ is Hack's exponent, from Hack's law (a power-law relationship between basin length and area) \cite{Hack57}.  This is combined with the equation $V_x = A_x \langle L_x \rangle $ to arrive at the relationship $V_x \propto A_x^{1+h}$.  This relationship is verified by their data on real river networks as well as OCNs, where (for both of which) $h$ is typically close to $0.57$ \cite{Maritan2002}.

If $h$ were to equal exactly $0.5$, this would imply isometric scaling of length and area in river basins, as $A_x$ would scale directly with $l_x^2$, and shape would be preserved.  This would line up with the prediction in \cite{size_and_form} that $V \propto A_x^{{D+1}/D}$ in the most efficient networks for $D = 2$.  However, since river basins tend to elongate with growth, getting proportionally narrower as they get larger, $h$ is usually slightly greater than $0.5$, and the lower bound predicted in \cite{size_and_form} is rarely reached.  

Shifting focus to three dimensions, we can look for evidence of a corresponding Hack's law for three dimensional OCNs (that is, evidence that area scales with length to some constant power).  If there is a relatively constant scaling exponent $h$, the analysis in \cite{Maritan2002} can be extended to three dimensions, since the steps taken there do not otherwise depend on dimension.  Thus, we might expect to see $V_x \propto A_x^{1+h}$ for some $h$; according to the analysis in \cite{size_and_form},  $V_x$ should scale at least as $A_x^{4/3}$, so we might expect $1/3$ to be a lower bound for $h$, with $h$ slightly greater than $1/3$ if three dimensional OCNs elongate in a similar way to two dimensional ones.  These predictions are borne out in the simulation results that follow.

\section{Experimental Results}

\subsection{Energy Scaling}

\begin{figure}[H]
	\centering
	\includegraphics[scale = .6]{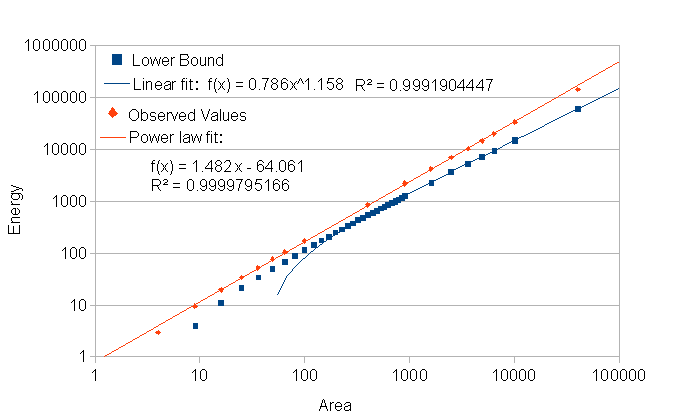}
	\caption{Log-log plot of area vs. energy for 2-dimensional OCNs, showing both analytical lower bound (lower, blue) and observed values (upper, red).}
	\label{fig:energies}
\end{figure}

Fig. \ref{fig:energies} shows how empirically observed minimum possible energy scales with area in two-dimensional OCNs.  The estimated line of best fit is shown, as well as the analytical lower bound derived in Section \ref{sec:proofs}.  The observed values of energy follow a power law with a small exponent, while the lower bound is essentially linear with respect to area.

\subsection{Length Scaling}

Scaling exponents $h$ where $l \propto A^h$ in two and three dimensions is shown in Table \ref{tab:length}.  The observed values of the exponents for two dimensions are within the bounds found in other studies \cite{Hack57, Maritan2002}.
For three dimensions, the exponent is also fairly consistent, though less so than in two dimensions.  It also deviates slightly more from the isometric value of 1/3 than the two dimensional version does from 1/2.  The fact that it is higher than 1/3 fulfils the hypothesis that three-dimensional basins elongate like two-dimensional ones.

\begin{table}[H]
	\centering
	\begin{tabular}{| l c |}
		\hline
		 basins analyzed &	$h$ \\ \hline
		2d whole basins &	0.6066895125 \\
		2d all subbasins &	0.5808520176 \\
		3d whole basins &	0.3645961454 \\
		3d all subbasins &	0.4230582779 \\
		\hline
	\end{tabular}
	\caption{Length scaling exponents}
	\label{tab:length}
\end{table}

\subsection{Volume Scaling}

Fig. \ref{fig:2dvol} shows the scaling of volume with area in two-dimensional OCNs, for whole networks and for all subbasins.    Note that the exponents are very similar, indicating that the scaling behavior is the same within basins as across different sizes of basins.  The exponents are both close to 1.57, as in \cite{Maritan2002}.

\begin{figure}[H]
	\begin{subfigure}[b]{0.5\textwidth}
		\centering
		\includegraphics[scale = .45]{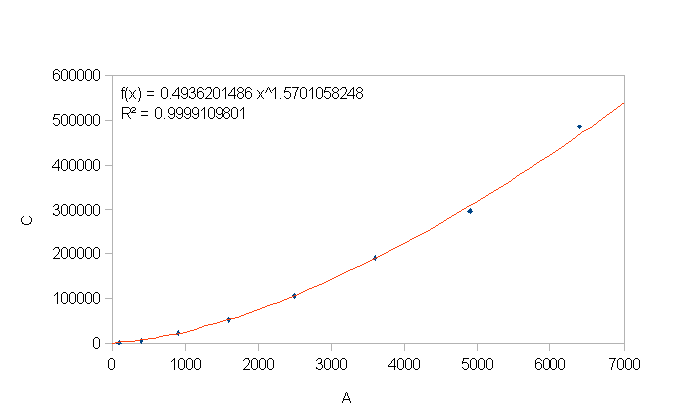}
		\caption{}
		\label{fig:wholes2}
	\end{subfigure}
	\begin{subfigure}[b]{0.5\textwidth}
		\centering
		\includegraphics[scale = .262]{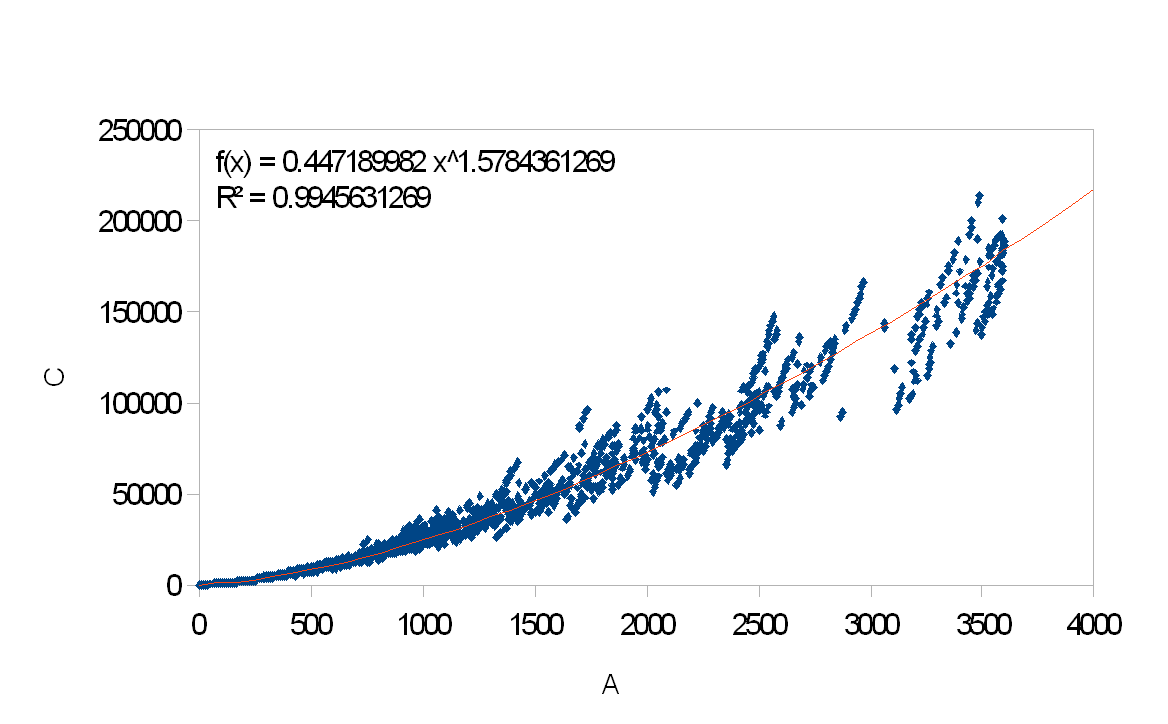}
		\caption{}
		\label{fig:2}
	\end{subfigure}
	\caption{Volume plotted with area for whole nxn basins (\ref{fig:wholes2}) and all subbasins (\ref{fig:2}) of 2-dimensional OCNs.}
	\label{fig:2dvol}
\end{figure}

\begin{figure}[H]
	\begin{subfigure}[b]{0.5\textwidth}
		\centering
		\includegraphics[scale = .45]{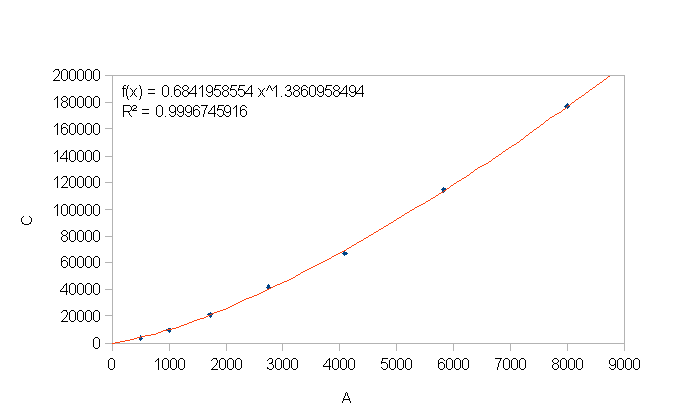}
		\caption{Whole nxn basins}
		\label{fig:wholes}
	\end{subfigure} 
	\begin{subfigure}[b]{0.5\textwidth}
		\centering
		\includegraphics[scale = .45]{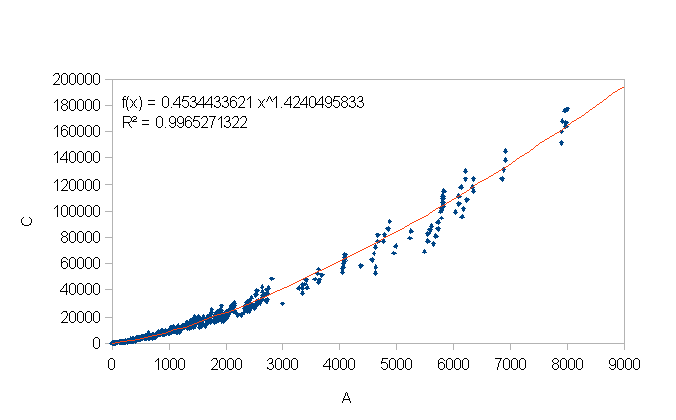}
		\caption{all subbasins}
		\label{fig:3dsubs}
	\end{subfigure} 
	\begin{subfigure}[b]{0.5\textwidth}
		\centering
		\includegraphics[scale = .45]{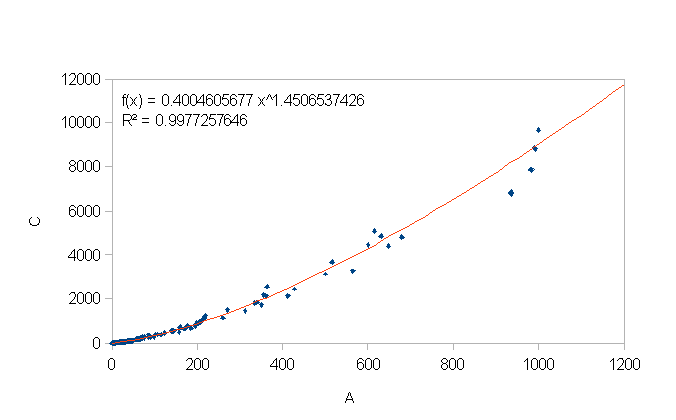}
		\caption{subbasins of 10x10x10 OCNs}
		\label{fig:3dsubs10}
	\end{subfigure} \quad
	\begin{subfigure}[b]{0.5\textwidth}
		\centering
		\includegraphics[scale = .45]{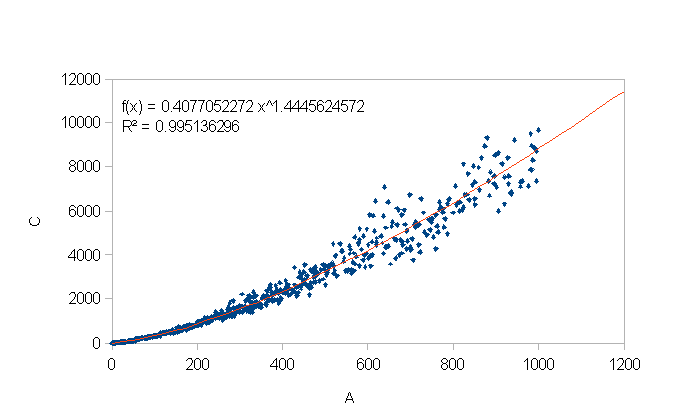}
		\caption{all subbasins with Area $\le$ 1000}
		\label{fig:3dsubssmall}
	\end{subfigure}
	\caption{Volume plotted with area for 3-dimensional OCNs.}
	\label{fig:3dvol}
\end{figure}

\begin{table}[H]
	\centering
	\begin{tabular}{| l c |}
		\hline
		basins analyzed & $\alpha$ \\ \hline
		whole networks &	1.386 \\
		all subbasins &	1.424 \\
		all subbasins of 20x20x20 network &	1.437 \\
		subbasins with area $>$ 1000 &	1.3776 \\
		subbasins of 20x20x20 network with area $>$ 1000 &	1.342 \\
		all subbasins with area $\le$ 1000 &	1.44456 \\
		subbasins of 20x20x20 network with area $\le$ 100 &	1.457 \\
		subbasins of 10x10x10 network &	1.4506 \\
		\hline
	\end{tabular}
	\caption{Volume scaling exponents for three-dimensional OCNs}
	\label{tab:3dvol}
\end{table}

The three-dimensional results for volume scaling are displayed in Fig. \ref{fig:3dvol} and summarized in Table \ref{tab:3dvol}.    The differences in exponents between the different data sets show higher scaling behavior in the sets of smaller basins; the difference between the exponent for whole basins (\ref{fig:wholes}) and all subbasins (\ref{fig:3dsubs}) is due to the greater concentration of smaller basins in the set of all subbasins, skewing the slope.  Note that the exponents in \ref{fig:3dsubs10} and \ref{fig:3dsubssmall} are similar, showing that scaling behavior in the subbasins is the same for the same distributions of area independent of the size of the encompassing whole basin.  In Section \ref{sec:proofs} we predicted that $\alpha$ would equal $1 + h$, using the analysis in \cite{Maritan2002}, and this relationship is clear in Tables \ref{tab:length} and \ref{tab:3dvol}.

The exponent $\alpha$ ranges from $\approx$1.34 to $\approx$1.46, with a higher exponent for collections of smaller basins.  These results are not inconsistent with the data for metabolic allometry in biological organisms, where a power-law fit to data from a wide range of sizes approximates that metabolic rate scales with mass to the power of 0.70 ($1/\alpha$) \cite{Savage2008Sizing}.  Not only is this value within the range found in the OCNs, but in organisms as well it is found that sets of smaller organisms yield a higher exponent than sets of larger ones \cite{Kolokotrones2010Curvature, Savage2004Predominance}.

\section{Alternative Models}

\subsection{Steiner Tree Model}

In addition to studying OCNs, we took a preliminary look at comparing OCNs with Steiner trees.

Steiner trees are minimum-weight trees that connect a given subset of nodes in a graph.  The specified nodes are called terminals.  Steiner trees may include non-terminal nodes; these are called Steiner points (or Steiner nodes).

In \cite{Chung1989Steiner}, Steiner trees built on a grid are considered.  The underlying graph is all points in a Euclidean plane, completely connected, and the terminals are the points of a grid.  We began comparing Steiner trees with OCNs by looking at the relative optimality of OCNs minimizing $H_{0.5}(s)$ and normal Steiner trees on grids of the same sizes.

In order to evaluate the energy of a Steiner tree, the tree needs to be directed; this can be done by choosing the outlet to be an arbitrary corner node and directing all the edges towards the outlet.

For the purposes of this study area was defined as the number of terminal nodes in the subtree rooted at a given node.  This way, the terminal nodes have exactly one unit of area associated with each of them, just as the grid nodes do in the OCN.

The OCNs here have realistic link lengths, where diagonal links are length $\sqrt{2}$, to make them comparable to the Steiner trees.

OCNs $s$ of size 2x2, 3x3, 4x4, and 5x5 were constructed and $H_{\gamma}(s)$ for each were compared with $H_{\gamma}(t)$ for Steiner trees $t$ of the same size.  The Steiner trees used were found in \cite{Chung1989Steiner}.  Figure \ref{fig:steinerpics} depicts the pairs of trees that were compared.  On the left of each pair is the Steiner tree, with open circles for the Steiner nodes, and on the right of each is the corresponding OCN.

\begin{figure}[H]
	\begin{subfigure}[b]{0.17\textwidth}
		\centering
		\includegraphics[scale = .5]{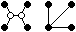}
		\caption{2x2}
		\label{fig:2x2pics}
	\end{subfigure}
	\begin{subfigure}[b]{0.21\textwidth}
		\centering
		\includegraphics[scale = .5]{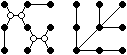}
		\caption{3x3}
		\label{fig:3x3pics}
	\end{subfigure}
	\begin{subfigure}[b]{0.27\textwidth}
		\centering
		\includegraphics[scale = .5]{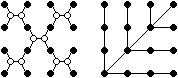}
		\caption{4x4}
		\label{fig:4x4pics}
	\end{subfigure}
	\begin{subfigure}[b]{0.33\textwidth}
		\centering
		\includegraphics[scale = .5]{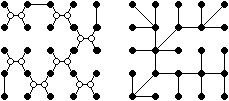}
		\caption{5x5}
		\label{fig:5x5pics}
	\end{subfigure}
	\caption{Steiner trees and OCNs for different nxn grids.}
	\label{fig:steinerpics}
\end{figure}

Since Steiner trees minimize total link length, it was clear that they would always have lowest energy for $\gamma = 0$.  For larger values of gamma, however, where area becomes a bigger factor than link length, we hypothesized that the OCNs would have a lower energy value than the Steiner trees.  Finding whether this was true, and where the crossover point would be, was the goal of the comparisons.

\subsection{Analytic Results}

\begin{figure}[H]
	\begin{subfigure}[b]{0.49\textwidth}
		\centering
		\includegraphics[scale = .75]{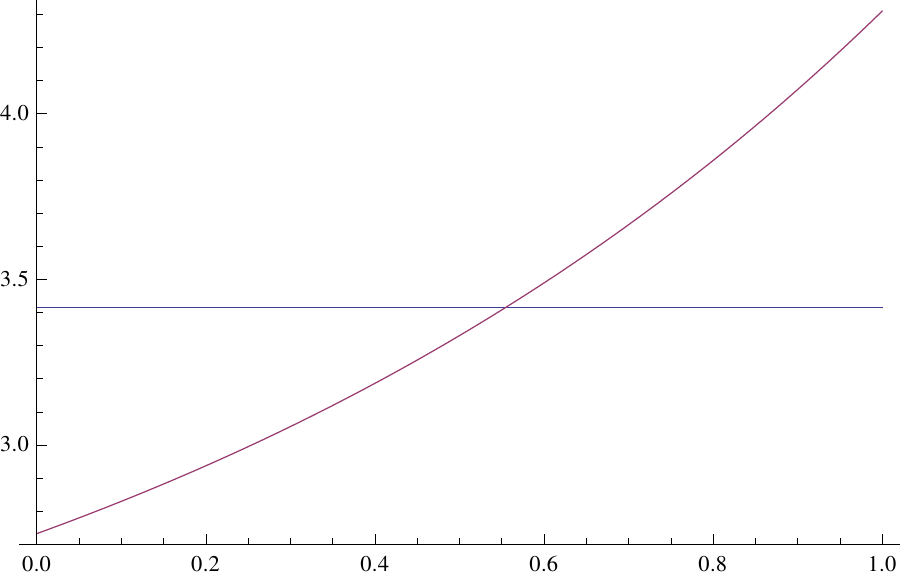}
		\caption{2x2}
		\label{fig:2x2graph}
	\end{subfigure}\quad
	\begin{subfigure}[b]{0.49\textwidth}
		\centering
		\includegraphics[scale = .75]{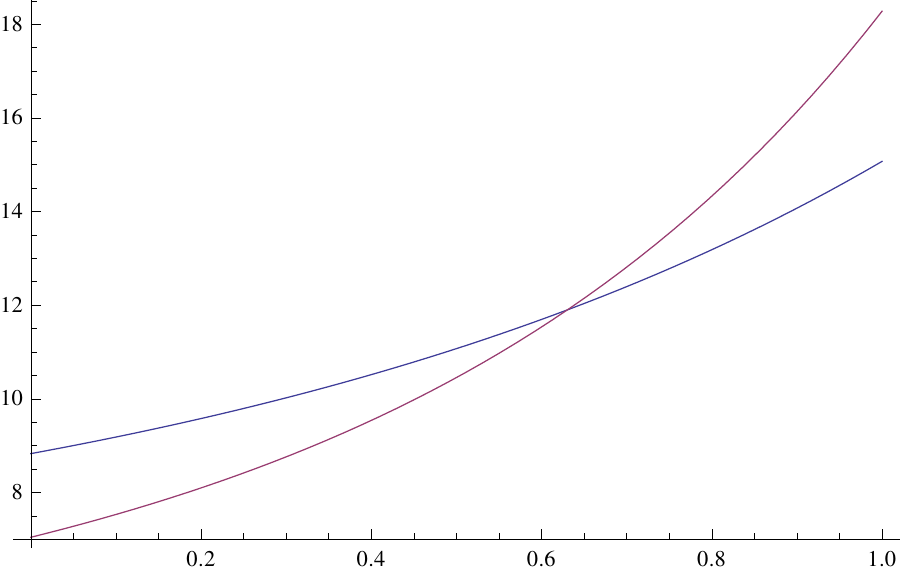}
		\caption{3x3}
		\label{fig:3x3graph}
	\end{subfigure}\\
	\begin{subfigure}[b]{0.49\textwidth}
		\centering
		\includegraphics[scale = .75]{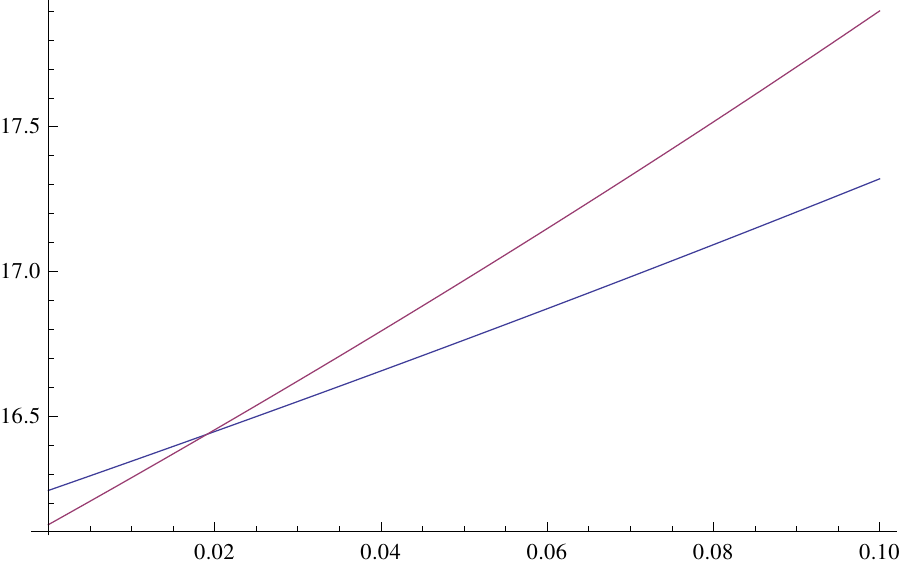}
		\caption{4x4}
		\label{fig:4x4graph}
	\end{subfigure}\quad
	\begin{subfigure}[b]{0.49\textwidth}
		\centering
		\includegraphics[scale = .75]{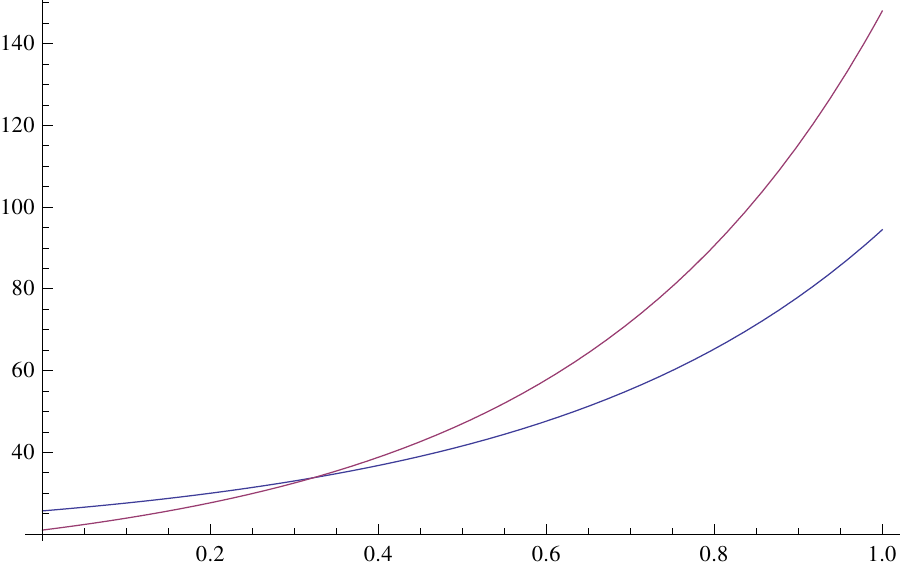}
		\caption{5x5}
		\label{fig:5x5graph}
	\end{subfigure}
	\caption{Plots of $H_{\gamma}(s)$ (blue) and $H_{\gamma}(t)$ (red), where $s$ is an OCN and $t$ a Steiner tree of the given size.  In each plot, the horizontal axis is $\gamma$ and the vertical axis is energy.}
	\label{fig:steinergraphs}
\end{figure}

Figure \ref{fig:steinergraphs} shows plots of the energy for corresponding Steiner trees and OCNs, $0 \le \gamma \le 1$.  As expected, the Steiner trees have lower energy than the OCNs for smaller $\gamma$, while for $\gamma$ past a certain crossover point the OCNs are better.  Table 3 
lists the crossover values of $\gamma$ for each size.

\begin{table}[H]
	\centering
	\begin{tabular}{| l l |}
		\hline
		size & crossover $\gamma$\\ \hline
		2x2 & 0.55433153188825812933\\
		3x3 & 0.63036154802337965422\\
		4x4 & 0.01907039679465522977\\
		5x5 & 0.32264708810531986705\\
		\hline
	\end{tabular}
	\caption{Crossover values of $\gamma$}
	\label{tab:crossovers}
\end{table}
\medskip

\subsection{Interpretation}

From the sizes studied, it is not possible to detect a trend in crossover points.

The next step would be to combine the two models, constructing Steiner-type trees that minimize $H_{\gamma}(s)$ for varying $\gamma$ and comparing them with real river networks and with normal OCNs.

\section{Conclusion}

We have looked at three dimensional OCNs as they compare with metabolite distribution networks in organisms.
In the organism context, volume is interpreted as volume of blood (which is directly proportional to the mass or volume of an organism) and area is proportional to the number of capillaries \cite{Savage2008Sizing}.  In this context, the grid used to create OCNs no longer represents fixed space (as it does in the two-dimensional/river context) and the amount of physical area directly associated with each node may not be the same over different body masses.  This would only have a linear effect on the relationship between area and mass, however, so this has no effect on the scaling exponent.

Though the three-dimensional networks do not actually look like the metabolic networks in organisms, we have shown that they have similar area/volume scaling behavior.  This adds credence to the idea that the observed metabolic scaling is a result of characteristics of the distribution network.

Another way to study the possible similarities between OCNs and metabolic networks would be to derive an energy functional for three dimensional OCNs from metabolic networks themselves, rather than using the one derived from rivers.  This could also include consideration of heterogeneity.  Heterogeneity in the river context is addressed in \cite{RodriguezIturbe1997Optimal} and \cite{Colaiori1997Analytical}; possibly this work could be extended to three dimensions.

\section{Acknowledgments}

The first author was supported by a Graduate Fellowship from Rensselaer Polytechnic Institute during the completion of this work.

\bibliographystyle{plain}
\bibliography{fullBib}

\begin{thebibliography}{10}

\bibitem{Agutter2011Analytic}
Paul~S. Agutter and Jack~A. Tuszynski.
\newblock {Analytic theories of allometric scaling}.
\newblock {\em The Journal of Experimental Biology}, 214(7):1055--1062, April
  2011.

\bibitem{Banavar2002Supplydemand}
Jayanth~R. Banavar, John Damuth, Amos Maritan, and Andrea Rinaldo.
\newblock {Supply-demand balance and metabolic scaling}.
\newblock {\em Proceedings of the National Academy of Sciences},
  99(16):10506--10509, August 2002.

\bibitem{size_and_form}
Jayanth~R. Banavar, Amos Maritan, and Andrea Rinaldo.
\newblock {Size and form in efficient transportation networks}.
\newblock {\em Nature}, 399(6732):130--132, May 1999.

\bibitem{Banavar2010General}
Jayanth~R. Banavar, Melanie~E. Moses, James~H. Brown, John Damuth, Andrea
  Rinaldo, Richard~M. Sibly, and Amos Maritan.
\newblock {A general basis for quarter-power scaling in animals}.
\newblock {\em Proceedings of the National Academy of Sciences},
  107(36):15816--15820, September 2010.

\bibitem{Brown2004Toward}
James~H. Brown, James~F. Gillooly, Andrew~P. Allen, Van~M. Savage, and
  Geoffrey~B. West.
\newblock {Toward a Metabolic Theory of Ecology}.
\newblock {\em Ecology}, 85(7):1771--1789, 2004.

\bibitem{Chung1989Steiner}
Fan Chung, Martin Gardner, and Ron Graham.
\newblock {Steiner Trees on a Checkerboard}.
\newblock {\em Mathematics Magazine}, 62:83--96, 1989.

\bibitem{Cieplak1998Models}
Marek Cieplak, Achille Giacometti, Amos Maritan, Andrea Rinaldo, Ignacio
  Rodriguez-Iturbe, and Jayanth~R. Banavar6.
\newblock {Models of Fractal River Basins}.
\newblock {\em Journal of Statistical Physics}, 91, 1998.

\bibitem{Colaiori1997Analytical}
F.~Colaiori, A.~Flammini, A.~Maritan, and Jayanth~R. Banavar.
\newblock {Analytical and numerical study of optimal channel networks}.
\newblock {\em Physical Review E}, 55(2):1298+, February 1997.

\bibitem{Dodds2001Reexamination}
Dodds, Rothman, and Weitz.
\newblock {Re-examination of the ``3/4-law'' of Metabolism}.
\newblock {\em Journal of Theoretical Biology}, 209(1):9--27, March 2001.

\bibitem{Dodds2010Optimal}
Peter~S. Dodds.
\newblock {Optimal Form of Branching Supply and Collection Networks}.
\newblock {\em Physical Review Letters}, 104(4):048702+, January 2010.

\bibitem{Hack57}
J.~T. Hack.
\newblock {Studies of longitudinal profiles in Virginia and Maryland}.
\newblock {\em U.S. Geological Survey Professional Paper}, 294-B:45--97, 1957.

\bibitem{Isaac2010Why}
Nick~J. Isaac and Chris Carbone.
\newblock {Why are metabolic scaling exponents so controversial? Quantifying
  variance and testing hypotheses.}
\newblock {\em Ecology letters}, 13(6):728--735, June 2010.

\bibitem{Kleiber1932Body}
Max Kleiber.
\newblock {Body size and metabolism}.
\newblock {\em Hilgardia}, 6:315--351, January 1932.

\bibitem{Kolokotrones2010Curvature}
Tom Kolokotrones, Van Savage, Eric~J. Deeds, and Walter Fontana.
\newblock {Curvature in metabolic scaling}.
\newblock {\em Nature}, 464(7289):753--756, April 2010.

\bibitem{ocean_balance}
Angel Lopez-Urrutia, Elena San~Martin, Roger~P. Harris, and Xabier Irigoien.
\newblock {Scaling the metabolic balance of the oceans}.
\newblock {\em Proceedings of the National Academy of Sciences},
  103(23):8739--8744, June 2006.

\bibitem{Maritan2002}
A.~Maritan, R.~Rigon, J.~R. Banavar, and A.~Rinaldo.
\newblock {Network allometry}.
\newblock {\em Geophysical Research Letters}, 29(11):1508+, June 2002.

\bibitem{Molnar1998Analysis}
Peter Moln\'{a}r and Jorge~A. Ram\'{\i}rez.
\newblock {An analysis of energy expenditure in Goodwin Creek}.
\newblock {\em Water Resources Research}, 34(7):null+, July 1998.

\bibitem{Moses2008Scaling}
Melanie~E. Moses, Stephanie Forrest, Alan~L. Davis, Mike~A. Lodder, and
  James~H. Brown.
\newblock {Scaling theory for information networks}.
\newblock {\em Journal of The Royal Society Interface}, 5(29):1469--1480,
  December 2008.

\bibitem{Renner2012Evaluation}
M.~Renner, R.~Seppelt, and C.~Bernhofer.
\newblock {E}valuation of water-energy balance frameworks to predict the
  sensitivity of streamflow to climate change.
\newblock {\em Hydrology and Earth System Sciences}, 16(5):1419--1433, May
  2012.

\bibitem{Rinaldo2006Trees}
A.~Rinaldo, J.~R. Banavar, and A.~Maritan.
\newblock {Trees, networks and hydrology}.
\newblock {\em Water Resources Research}, 42, 2006.

\bibitem{Rinaldo1992Minimum}
Andrea Rinaldo, Ignacio RodriguezIturbe, Riccardo Rigon, Rafael~L. Bras, Ede
  Ijjasz-Vasquez, and Alessandro Marani.
\newblock {Minimum energy and fractal structures of drainage networks}.
\newblock {\em Water Resources Research}, 28(9):null+, 1992.

\bibitem{RodriguezIturbe2011Metabolic}
Ignacio Rodriguez-Iturbe, Kelly~K. Caylor, and Andrea Rinaldo.
\newblock {Metabolic principles of river basin organization}.
\newblock {\em Proceedings of the National Academy of Sciences},
  108(29):11751--11755, July 2011.

\bibitem{RodriguezIturbe1997Optimal}
Ignacio Rodriguez-Iturbe and Andrea Rinaldo.
\newblock {\em Fractal River Basins: Chance and Self-Organization}, chapter~4.
\newblock Cambridge University Press, 1997.

\bibitem{Savage2004Predominance}
V.~M. Savage, J.~F. Gillooly, W.~H. Woodruff, G.~B. West, A.~P. Allen, B.~J.
  Enquist, and J.~H. Brown.
\newblock {The predominance of quarter-power scaling in biology}.
\newblock {\em Functional Ecology}, 18(2):257--282, 2004.

\bibitem{Savage2008Sizing}
Van~M. Savage, Eric~J. Deeds, and Walter Fontana.
\newblock {Sizing Up Allometric Scaling Theory}.
\newblock {\em PLoS Comput Biol}, 4(9):e1000171+, September 2008.

\bibitem{wbe}
Geoffrey~B. West, James~H. Brown, and Brian~J. Enquist.
\newblock {A General Model for the Origin of Allometric Scaling Laws in
  Biology}.
\newblock {\em Science}, 276(5309):122--126, April 1997.

\bibitem{West1999Fourth}
Geoffrey~B. West, James~H. Brown, and Brian~J. Enquist.
\newblock {The Fourth Dimension of Life: Fractal Geometry and Allometric
  Scaling of Organisms}.
\newblock {\em Science}, 284(5420):1677--1679, June 1999.

\bibitem{White2011Manipulative}
Craig~R. White, Michael~R. Kearney, Philip~G. Matthews, Sebastiaan~A. Kooijman,
  and Dustin~J. Marshall.
\newblock {A manipulative test of competing theories for metabolic scaling.}
\newblock {\em The American naturalist}, 178(6):746--754, December 2011.

\end{thebibliography}

\endgroup
\end{document}